\documentclass[12pt,a4paper]{article}

\usepackage[a4paper, total={7in, 9in}]{geometry}
\usepackage{multicol,multirow}
\usepackage{amsmath,amssymb,amsfonts}
\usepackage{mathrsfs}
\usepackage{amsthm}
\usepackage{rotating}
\usepackage{appendix}
\usepackage[numbers]{natbib}
\usepackage{ifpdf}
\usepackage[T1]{fontenc}
\usepackage{newtxtext}
\usepackage{newtxmath}
\usepackage{textcomp}

\newtheorem{theorem}{Theorem}[section]
\newtheorem{proposition}[theorem]{Proposition}
\newtheorem{corollary}[theorem]{Corollary}
\newtheorem{lemma}[theorem]{Lemma}
\theoremstyle{definition}

\newtheorem{problem}[theorem]{Problem}
\numberwithin{equation}{section}

\newcommand{\Z}{{\mathbb Z}}
\newcommand{\N}{{\mathbb N}}

\title{On Diophantine graphs}
\author{
Gerg\H{o} Batta\thanks{Institute of Mathematics, University of Debrecen, batta.gergo@science.unideb.hu}, 
Lajos Hajdu\thanks{Institute of Mathematics, University of Debrecen and HUN-REN DE Equations, Functions, Curves and their Applications Research Group, hajdul@science.unideb.hu}, 
Andr\'as Pongr\'acz\thanks{HUN-REN Alfr\'ed R\'enyi Institute of Mathematics, ORCID 0000-0002-2771-8974, pongracz.andras@renyi.hu, andras.pong@gmail.com}
}

\date{June 9, 2024}

\begin{document}

\maketitle

\abstract{Diophantine tuples are of ancient and modern interest, with a huge literature. 
In this paper we study Diophantine graphs, i.e.,  finite graphs whose vertices are distinct positive integers, and two vertices are linked by an edge if and only if their product increased by one is a square. 
We provide various results for Diophantine graphs, including extendability properties, lower- and upper bounds for the maximum number of edges and chromatic numbers.}


\section{Introduction}

A set $\{a_1,\dots,a_n\}$ of distinct positive integers is called a Diophantine $n$-tuple if $a_ia_j+1$ is a square for all $1\leq i<j\leq n$. 
Diophantine tuples are of ancient and modern interest, with a huge literature. 
We only mention a few results. 
As the first 'modern' result, Baker and Davenport \cite{BaDe} showed that $\{1,3,8\}$ can be extended to a Diophantine quadruple only by adjoining $120$. 
In particular, $\{1,3,8\}$ cannot be extended to a Diophantine quintuple. 
Later, the latter result was extended by Dujella and Peth\H{o} \cite{DP}, who showed that already the pair $\{1,3\}$ cannot be extended to a Diophantine quintuple. 
Dujella \cite{Duj04} proved that there are no Diophantine sextuples and there are only finitely many Diophantine quintuples, while He, Togb\'e and Ziegler \cite{HTZ19} showed that there are no Diophantine quintuples. 
For other related research directions and results, we refer to the above-mentioned papers and the references given there, the homepage \cite{Dujhomepage} maintained by Dujella, and an excellent, brand new book \cite{Dujbook} of Dujella.

In the present paper we study Diophantine graphs. 
Given a finite set $V$ of positive integers, the induced Diophantine graph $D(V)$ has vertex set $V$, and two numbers in $V$ are linked by an edge if and only if their product increased by one is a square. 
A finite graph $G$ is a Diophantine graph if it is isomorphic to $D(V)$ for some finite set $V\subseteq \mathbb{N}$; such a set $V$ is called a witness of $G$, and $D(V)$ is a representation of $G$ as a Diophantine graph. 
Already one can find several related results in the literature. 
To start with, observe that the result of He, Togb\'e and Ziegler \cite{HTZ19} implies that $K_5$ is a ``forbidden'' subgraph, that is, a graph containing $K_5$ is not a Diophantine graph. 
(Here and later on $K_t$ $(t\geq 1)$ stands for the complete graph on $t$ vertices.) 
Further, Dujella \cite{Duj08} gave upper bounds for the number of the $K_t$ subgraphs of $D(V)$ for $t=2,3,4$ and $V=\{1,\dots,N\}$. 
Note that the motivation of \cite{Duj08} is to study Diophantine pairs, triples and quadruples among the first $N$ positive integers. 
Another result we mention is Theorem 2 of Bugeaud and Gyarmati \cite{BGy04}, who (using classical tools from extremal graph theory) among others proved that the number of edges in a Diophantine graph on $N$ vertices is bounded by $0.4N^2$. 
We mention that in \cite{BGy04} various further results are given, related to the cases where the shifted products $a_ia_j+1$ are $k$-th powers with $k\geq 3$, or are arbitrary (not necessarily equal) powers. 
Finally, we mention a recent paper of Yip \cite{yip}, concerning bipartite graphs in the case where the shifted products are $k$-th powers with $k\geq 3$. 
The cases $k=2$ and $k\geq 3$ are rather different in nature: for $k\geq 3$ the problem can be reduced to Thue equations (which are known to have finitely many solutions, and hence provide a strong tool to use right away), while for $k=2$ we have to deal with Pell-type equations (having infinitely many solutions).

We provide results in three directions. 
The first direction concerns the extendability of Diophantine graphs. 
We show that it is always possible to attach an isolated new vertex, or a new vertex linked to precisely one, arbitrary old vertex, or (under a simple, necessary condition) a new vertex linked to precisely two, arbitrary old vertices of a Diophantine graph. 
To prove our statements we shall need to combine various tools related to Pell-equations, simultaneous Pell-equations,  and elliptic equations. 
Then we give lower- and upper bounds for the maximum number of edges in a Diophantine graph on $N$ vertices. 
From an aforementioned result of Dujella \cite{Duj08} we know that this number is at least $\frac{6}{\pi^2}N\log N+\Theta(N)$. 
(Throughout the paper, $O, \Omega$, and $\Theta$ are used in the standard way as defined by Knuth: $f(n)=O(g(n))$ if $|f(n)|\leq C|g(n)|$ with some positive constant $C$ for all sufficiently large $n$; $f(n)=\Omega(g(n))$ if $g(n)=O(f(n))$, and finally, $f(n)=\Theta(g(n))$ if $f(n)=O(g(n))$ and $f(n)=\Omega(g(n))$.) 
We shall prove that this is not best possible, and improve the lower bound to $N(\log N)^{2\log 2-\varepsilon}$. 
To this end, we shall need advanced tools from prime number theory: namely, Hardy-Ramanujan type results of Sathe \cite{Sat53} and Selberg \cite{Sel54} on the asymptotic behavior of the function $\pi(x,k)$. 
As a simple combination of Tur\'an's famous theorem concerning cliques and the result of He, Togb\'e and Ziegler \cite{HTZ19} we can slightly improve upon the upper bound of Bugeaud and Gyarmati \cite{BGy04} for the number of edges. 
However, we strongly believe that Diophantine graphs are sparse. 
In this direction we can only prove a result under some assumptions, and we shall also need to assume the validity of a conjecture of Szpiro (see \cite{Sz}, and also Conjecture 0.4 in \cite{hs}) and N\'eron's conjecture \cite{ne} (and/or the related heuristics of Park, Poonen, Voight and Wood \cite{PPVW}) supporting that the rank of elliptic curves is absolutely bounded. 
Further, we have to use deep results of Hindry and Silverman \cite{hs} about the number of integral points on elliptic curves, and a result from extremal graph theory related to forbidden bipartite graphs also plays an important role here. 
Finally, based upon combinatorial and prime theoretical considerations, we are able to exhibit a Diophantine graph with chromatic number $5$. 
Note that as we know that $K_4$ is a Diophantine graph but $K_5$ is not, this is the first non-trivial step in this direction.

The structure of the paper is the following. 
In the next section we provide all our main results. 
Then we give their proofs (together with the necessary background) in separate sections. 
We conclude the paper with some remarks and open problems.

\section{Main results}

Now we provide our main results, arranged into distinct subsections.

\subsection{Extendability of Diophantine graphs}
By combining the theory of Pell- and simultaneous Pell equations with various other tools, we study the extensions of Diophantine graphs with a fixed representation. 

As we will see, it is always possible to find a new isolated vertex or a vertex which is linked to exactly one vertex of a given Diophantine graph $D(V)$. 
In fact, there are always infinitely many appropriate positive integers to solve these problems.  
However, to extend a Diophantine graph $D(V)$ by a vertex which is linked to exactly two given vertices in $V$ is more problematic. 
If the square-free part of two different numbers $v_1, v_2\in V$ coincide, than there are only finitely many common neighbors $w$ of $v_1$ and $v_2$ in $\mathbb{N}$. 
Indeed, the condition means that there are positive integers $a\neq b$ such that $v_1a^2=v_2b^2$. 
Thus if $wv_1+1=r^2$ and $wv_2+1=s^2$, then $(bs)^2-b^2=wv_2b^2=wv_1a^2=(ar)^2-a^2$, making $(ar)^2-(bs)^2=a^2-b^2$. 
As there are only finitely many pairs of square numbers with given nonzero difference $a^2-b^2$, there can only be finitely many solutions $w,r,s$ to the system of equations. 
(Sometimes there is no such solution, e.g., if $(v_1,v_2)=(1,4)$. For $(v_1,v_2)=(1,16)$, there is exactly one common neighbor $w=3$, arising from the unique solution $(w,r,s)=(3,2,7)$.)

\begin{theorem}
\label{thm1}
Let $V=\{v_1,\dots,v_n\}\subseteq \mathbb{N}$. 
Then each of the following conditions is satisfied by infinitely many positive integers $w$ whose square-free part differs from that of $v_k$ for all $1\leq k\leq n$: 
\begin{itemize}
\item[i)] $w$ is an isolated vertex of $D(V\cup\{w\})$,
\item[ii)] $w$ is linked in $D(V\cup\{w\})$ to exactly one arbitrarily prescribed vertex $v_i\in V$,
\item[iii)] $w$ is linked in $D(V\cup\{w\})$ to exactly two prescribed vertices $v_i,v_j\in V$ with different square-free part.
\end{itemize}
\end{theorem}
 
Note that a classical result of Thue~\cite{Thu09}, made effective by Baker~\cite{Baker68} implies the following. 

\begin{proposition}\label{prop:threeneigh}
Any three positive integers have only finitely many common neighbors in $\mathbb{N}$. 
\end{proposition}

For example, the triple $1,2,3$ does not have any common neighbor. 
This can be verified by using the procedure {\tt IntegralPoints} of {\tt{Magma}}~\cite{magma}
(based upon results of Gebel, Peth\H{o}, Zimmer \cite{GPZ94} and Stroeker,
Tzanakis \cite{ST94} obtained independently), applied for the elliptic curve
$X^3-351X+1890=Y^2$, derived from $(x+1)(2x+1)(3x+1)=y^2$ by the substitution
$(X,Y)=(54x+33,162y)$.
The general result, Proposition~\ref{prop:threeneigh} follows by using the same technique: there are always finitely many integer points on such an elliptic curve according to \cite{Thu09, Baker68}. 
Thus Theorem~\ref{thm1} cannot be generalized to three (or more) prescribed vertices. 
Moreover, according to a result of Baker and Davenport \cite{BaDe}, the only common neighbor of $1,3,8$ in $\mathbb{N}$ is 120. 
As 7 is also linked to 120, there is no way to extend $1,3,7,8$ by a new vertex that is linked to $1,3,8$ but not to $7$.

\subsection{Bounds for the number of edges in Diophantine graphs}

Our next theorems concern the edge density in Diophantine graphs, or putting it in another way, the number of edges $e(D(V))$ of Diophantine graphs $D(V)$ on $n$ vertices. 
In fact, we are interested in $\max\limits_{|V|=n} e(D(V))$ for $n\in\N$ or the order of magnitude of this function.

We start with giving lower bounds. 
First we note that Theorem \ref{thm1} easily implies the existence of graphs $D(V)$ with $|V|=n$ and $e(D(V))= \Omega(n)$. 
Further, Theorem 1 of Dujella \cite{Duj08} yields the existence of such graphs with
$$
e(D(V))= \frac{6}{\pi^2}n\log n+\Theta(n).
$$
In fact, Dujella proved that the above assertion holds for the graph $D(V_N)$ induced by $V_N:=\{1,\dots,N\}$.

By omitting elements from $V_N$ of ``small'' degree, relying on deep Hardy-Ramanujan type results of Sathe \cite{Sat53} and Selberg \cite{Sel54} concerning the function $\pi(x,k)$, we obtain the following statement.

\begin{theorem}
\label{thm:densegr}
For any $\varepsilon>0$ there exists an arbitrarily large $n$ and a Diophantine graph $D(V)$ with $n$ vertices such that
$$
e(D(V))>n(\log n)^{2\log 2-\varepsilon}.
$$
\end{theorem}

Now we turn to upper bounds for $\max\limits_{|V|=n} e(D(V))$  for $n\in\N$. 
Our first theorem in this direction is obtained by a simple combination of the main result of He, Togb\'e and Ziegler \cite{HTZ19} and a seminal theorem of Tur\'an \cite{Tu41} concerning forbidden complete subgraphs. 
However, we formulate it as a theorem because of its strategic importance. 
Note that the statement is an improvement of Theorem 2 of Bugeaud and Gyarmati \cite{BGy04} with an upper bound $0.4n^2$, obtained by similar tools.

\begin{theorem}
\label{thm3}
For any Diophantine graph $D(V)$ with $|V|=n$ we have
$$
e(D(V))\leq \frac{3}{8}n^2.
$$
\end{theorem}

We strongly suspect that a much better upper bound should be valid, and $\max\limits_{|V|=n} e(D(V))=o(n^2)$ should hold. 
In order to find a sub-linear upper estimate for the edge density of Diophantine graphs, it is inevitable to answer the following question: is $K_{t,t}$, the complete equipartite bipartite graph on $n=2t$ vertices, the subgraph of a Diophantine graph for all $t\in \mathbb{N}$? 
Equivalently, does there exist for all $t\in \mathbb{N}$ some positive integers $a_1, \ldots, a_t$ and $b_1, \ldots, b_t$ such that $a_ib_j+1$ is a perfect square for all $1\leq i,j \leq t$? 
If the answer is positive, then it yields an infinite sequence of Diophantine graphs with $n$ vertices and at least $n^2/4$ edges. 
If the answer is negative, with counterexample $t\in \mathbb{N}$, then according to a classical result in extremal graph theory~\cite{KST54}, there is an $O(n^{2-1/t})$ upper bound for the number of edges in Diophantine graphs on $n$ vertices. 
However, this problem seems to be out of reach with state-of-the-art methods: see Problem 1.4 of Dujella \cite{Dujproblems}, which asks whether already $K_{3,3}$ is a Diophantine graph or not. 
(Note that as a simple consequence of Theorem \ref{thm1}, we obtain that $K_{2,t}$ is a Diophantine graph for any $t$.) 
We present two related positive results. 

As we already mentioned, any three positive integers have only finitely many common neighbors in $\mathbb{N}$ due to \cite{Thu09, Baker68}; see Proposition~\ref{prop:threeneigh}.
That argument provides an upper bound for the number of common neighbors in terms of the three positive integers, but not a uniform upper bound. 
If such a uniform upper bound could be provided, that would settle the question concerning $K_{t,t}$. 
We show a stronger result in this direction than Proposition~\ref{prop:threeneigh}. 
This next theorem is strongly conditional, we need to assume two deep conjectures. 
The first one is due to Szpiro (see \cite{Sz}, and also Conjecture 0.4 in \cite{hs}). 
We only state it in the special case when the base field is the field of rational numbers: 
for all $\varepsilon>0$ there are only finitely many elliptic curves $E$ over $\mathbb Q$ satisfying
\begin{equation}
\label{szpiro}
\frac{\log |D_E|}{\log C_E}\geq 6+\varepsilon,
\end{equation}
where $D_E$ is the minimal discriminant and $C_E$ is the conductor of $E$. The second conjecture concerns the ranks of elliptic curves over $\mathbb Q$. For a long time, it has been widely believed that there is no absolute bound for them. However, recent heuristics of Park, Poonen, Voight and Wood \cite{PPVW} suggest that possibly the opposite can be true. Here we assume this option, which has already been predicted (among others) by N\'eron \cite{ne}:
\begin{equation}
\label{neron}
\text{the ranks of elliptic curves over}\ {\mathbb Q}\ \text{are uniformly bounded}.
\end{equation}
For an integer $d$, let $\omega(d)$ be the number of distinct prime factors of $d$.

\begin{theorem}
\label{thm4}
Let $a,b\in V$, and suppose that $D(V)\cong K_{t,t}$ with $t\geq 1$, such that $a,b$ belong to the same vertex class. Then, assuming Szpiro's conjecture \eqref{szpiro} and N\'eron's conjecture \eqref{neron}, there exists a constant $C=C(\omega(ab(a-b)))$ such that $t<C$.
\end{theorem}

\subsection{Chromatic number of Diophantine graphs}

Finally, we study the chromatic number of Diophantine graphs. 
As the clique number, i.e., the size of the largest clique, of Diophantine graphs is at most four (see \cite{HTZ19}), it is plausible to ask whether all Diophantine graphs are four-colorable. 
Our next result shows that it is not the case. As we shall see, already this step requires considerable effort.

\begin{theorem}
\label{thm5}
Let $V=\{${\rm 1, 3, 8, 120, 2, 4, 12, 20, 24, 6, 22, 92, 204, 420, 36, 78, 84, 140, 210, 360, 364, 560, 60, 14, 40, 136, 220, 312, 33, 9, 10, 52, 56, 728, 11, 48, 90, 168, 408, 840, 5, 7, 28, 30, 34, 35, 46, 70, 88, 132, 180, 240, 2184, 280, 16, 21, 32, 44, 156, 816, 380, 13, 39, 72, 80, 96, 462, 528, 1140, 2380, 23, 102, 105, 110, 152, 264, 456, 858, 2520, 1365}$\}$. 
Then the graph $D(V)$ has chromatic number five.
\end{theorem}
We note that this $80$-element set induces the smallest Diophantine graph we know of that cannot be colored by four colors. 
We do not claim that this is the smallest example, but it is minimal: if any of the vertices is omitted from $V$, the resulting graph is four-colorable.

\section{Proof of the results on extensions of Diophantine graphs}

In this section we prove Theorem \ref{thm1}. 
In fact, the statement is an immediate consequence of the following three lemmas.

\begin{lemma}
\label{lem1}
Let $v_1,\dots,v_n$ be different positive integers. Then there exist infinitely many positive integers $w$ such that $v_i w+1$ is not a square, and the square-free part of $w$ and $v_i$ are different for all $i=1,\dots,n$.
\end{lemma}

\begin{proof}
Consider the system of congruences
\begin{align*}
v_i x + 1 &\equiv p_i \qquad (\bmod \; p_i^2)\\
x &\equiv q \qquad (\bmod \; q^2) 
\end{align*}
in integers $x$, where $p_i$ $(i=1,\dots,n)$ are different primes with $p_i\nmid a_i$, and $q$ is a prime different from all $p_i$ with $q\nmid v_i$ $(i=1,\dots,n)$. 
By the Chinese Remainder Theorem this congruence system has infinitely many positive solutions. Choosing $w$ among these solutions, our claim follows.
\end{proof}

\begin{lemma}
\label{lem2}
Let $v_1,\dots, v_{n}$ be different positive integers, and let $i\in\{1,\dots,n\}$ be fixed. Then there exist infinitely many positive integers $w$ such that $v_i w+1$ is a square, $v_jw +1$ is not a square for any $j\in\{1,\dots,n\}$ with $j\neq i$, and the square-free part of $w$ is different from those of the $v_\ell$ $(\ell=1,\dots,n)$.
\end{lemma}

\begin{proof}
Let $X$ be a positive real number which is large enough. (The precise magnitude of $X$ will become clear later.) Further, let $q$ be the smallest prime such that $q\nmid v_1\cdots v_n$, and consider the Pell equation
\begin{equation}
\label{eq2}
z^2-qv_iy^2=1.
\end{equation}
Let $z_0,y_0$ be the fundamental solution of \eqref{eq2}. Then $y_0\neq 0$; let $y_0=q^ty_1$ with $t\geq 0$ and $q\nmid y_1$. Consider the system of congruences
\begin{align*}
\label{congsys}
x &\equiv 1 \qquad (\bmod \; v_i) \\
x &\equiv z_0 \qquad (\bmod \; q^{2t+2}) 
\end{align*}
in integers $x$. 
By the Chinese Remainder Theorem it has precisely one solution modulo $M:=q^{2t+2}v_i$. 
Observe that $M$ depends only on $V:=\{ v_1,\dots,v_n\}$. 
The number of solutions $x>1$ of the above system up to $X$ is at least $\lfloor X/M\rfloor-1$. 
Moreover, if we take such a solution $x$ and define $w:=(x^2-1)/v_i$, then we clearly have that $w$ is a positive integer with $v_i w+1$ being a square, and by \eqref{eq2} also that
$$
v_i w+1\equiv z_0^2\equiv qv_iy_0^2+1\pmod{q^{2t+2}}
$$
whence
$$
w\equiv q^{2t+1}y_1^2\pmod{q^{2t+2}}.
$$
The last congruence shows that the square-free part of $w$ is divisible by $q$, so it is different from all those of the $v_\ell$ $(\ell=1,\dots,n)$. So it is enough to show that for 'most' of such integers $w$, we have that $v_j w+1$ is not a square for $j\in\{1,\dots,n\}$, $j\neq i$. For this, assume to the contrary, that together with
\begin{equation}
\label{xeq}
v_i w+1=x^2
\end{equation}
we also have
$$
v_j w+1 = r^2
$$
for some $j\in\{1,\dots n\}$ with $j\neq i$. Combining these equalities, we easily get
\begin{equation}
\label{pelleq}
v_ir^2 - v_jx^2 = v_i-v_j
\end{equation}
which is a generalized Pell equation. 
Standard results imply that the number of integers $x$ with $1<x\leq X$ satisfying \eqref{pelleq} is at most $C'\log X$, where $C'$ is a constant depending only on $V$. 
(This follows easily e.g. from formula (2.18) on p. 145 of \cite{HaSa}, which concerns the same equation.) 
So there is a constant $C$ such that the number of positive integers $x$ up to $X$ satisfying \eqref{xeq} such that $v_j w+1$ is a square with $w=(x^2-1)/v_i$ and with some $j\in\{1,\dots,n\}$, $j\neq i$, is at most $C\log X$. 
From this, letting $X$ tend to infinity, our claim follows.
\end{proof}

\begin{lemma}
\label{lem3}
Let $v_1,\dots,v_n$ be different positive integers, and let $i,j$ be fixed distinct indices from $\{1,\dots,n\}$ such that the square-free part of $v_i$ and $v_j$ are different. 
Then there exist infinitely many positive integers $w$ such that $v_i w+1$ and $v_j w+1$ are squares, $v_\ell w +1$ is not a square for $\ell\in\{1,\dots,n\}$, $\ell\neq i,j$, and the square-free part of $w$ is different from all those of the $v_m$ $(m=1,\dots,n)$.
\end{lemma}

\begin{proof}
First we show that the system of equations 
\begin{align*}
v_i x+1 &= r_i^2 \\
v_j x+1 &= r_j^2
\end{align*}
has infinitely many solutions in positive integers $x,r_i,r_j$. From these equations we easily obtain that
$$
v_i'r_j^2 - v_j'r_i^2 = v_i'-v_j'
$$
where $v_i=dv_i'$ and $v_j=dv_j'$ with $d=\gcd(v_i,v_j)$. Hence we get
\begin{equation}
\label{Pell eq}
X^2 - v_i'v_j' Y^2 = v_i'(v_i'-v_j')
\end{equation}
where $X=v_i'r_j$ and $Y=r_i$. Since $v_i'v_j'$ is not a square, this is a generalized Pell equation. Observe that a solution to it is given by $(X,Y)=(v_i',1)$.
Let $\varepsilon=\mu+\nu\sqrt{v_i'v_j'}$ $(\mu,\nu\in\Z)$ be the fundamental unit of $\mathcal{O}:=\Z[\sqrt{v_i'v_j'}]$ of norm $1$. 
We see that \eqref{Pell eq} has infinitely many solutions of the form $(X,Y)=(X_t,Y_t)$, where $X_t,Y_t$ are defined by
$$
X_t+Y_t\sqrt{v_i'v_j'}=\left(v_i'+\sqrt{v_i'v_j'}\right)\left(\mu+\nu\sqrt{v_i'v_j'}\right)^t
$$
with $t\geq 0$. 
From this, we easily get that $v_i'\mid X_t$. 
Let $t_0$ be the order of $\varepsilon$ in $\mathcal{O}$ modulo $v_i$. Then for any $t$ with $t_0\mid t$ we have
$$
X_t+Y_t\sqrt{v_i'v_j'}\equiv v_i'+\sqrt{v_i'v_j'}\pmod{v_i},
$$
implying $v_i\mid (Y_t-1)$ in this case. So if we take such a $t$, then writing
$$
x:=\frac{Y_t^2-1}{v_i}=\frac{(X_t/v_i')^2-1}{v_j},
$$
we see that $x$ satisfies the original system of equations. 
That is, we obtained infinitely many solutions for that system, indeed. 
Now we show that the square-free parts of these solutions can come from an infinite set. 
Assume to the contrary that all such solutions can be written as $x=sz^2$, where $s$ comes from a finite set. 
Then multiplying the original equations, we obtain
$$
v_iv_js^2z^4+(v_i+v_j)sz^2+1=(r_ir_j)^2.
$$
For a fixed $s$, this is a hyperelliptic equation (in $z$ and $r_ir_j$ as integral variables). 
According to Baker's classical result \cite{Baker69}, it has only finitely many solutions, provided that the polynomial on the left hand side has simple roots in $z$. 
Since the discriminant of that polynomial is $16v_iv_js^6(v_i-v_j)^4$, and $v_i\neq v_j$, our claim follows. 
So we have infinitely many positive integers $w$ such that $v_i w+1$ and $v_j w+1$ are both squares, and the square-free parts of these $w$ are all different. 
Now we only need to show that for infinitely many such $w$, $v_\ell w+1$ is not a square for $\ell\in\{1,\dots,n\}$, $\ell\neq i,j$. 
To this end, consider the system of equations 
\begin{align*}
v_ix+1 &= r_i^2 \\
v_jx+1 &= r_j^2 \\
v_\ell x+1 &= r_\ell^2
\end{align*} 
in positive integers $x,r_i,r_j,r_\ell$, where $\ell\in\{1,\dots,n\}$, $\ell\neq i,j$. This yields 
$$
(v_ix+1)(v_jx+1)(v_\ell x+1)=(r_ir_jr_\ell)^2,
$$
which is an elliptic equation (in $x$ and $r_ir_jr_\ell$ as integral unknowns). 
By a classical results of Baker \cite{Baker68} (using that $v_i,v_j,v_\ell$ are distinct positive integers) this equation has only finitely many solutions. 
\end{proof} 

\begin{corollary}
Every finite graph has a Diophantine subdivision. 
In particular, every finite graph is a minor of a Diophantine graph. 
\end{corollary}
\begin{proof}
Let $G$ be a finite graph with $n$ vertices. 
According to Lemma~\ref{lem1} we can choose a set $V$ of $n$ positive integers iteratively, each with a different square-free part, such that $D(V)$ is an empty graph. 
Let us identify the vertex set of $G$ with $V$. 
For each edge of $G$, we can extend $V$ by a new number that is linked to exactly the two endpoints of the edge. 
This can also be done iteratively: in each step, we choose a new number that is not linked to any number in the extended set other than the two endpoints of the given edge. 
\end{proof}

\begin{corollary}\label{cor:3reg}
The following are equivalent. 
\begin{enumerate}
\item Every finite graph with maximum degree at most 3 can be represented as a Diophantine graph such that the numbers associated to vertices of degree at most 2 have different square-free part. 
\item Every finite graph with maximum degree at most 3 is Diophantine. 
\item Every finite 3-regular graph is Diophantine. 
\end{enumerate}
\end{corollary}
\begin{proof}
The implications $(1)\Rightarrow (2)\Rightarrow (3)$ are trivial. 

We prove $(3)\Rightarrow (1)$ by induction on the cardinality $n$ of the vertex set of the graph $G$; the assertion is trivial if $n=1$. 
Now assume that $n>1$, item $(3)$ holds, and item $(1)$ holds for all graphs with at most $n-1$ vertices. 
If $G$ is 3-regular, then it has a Diophantine representation by item $(3)$: the condition about the square-free part of vertices of degree at most 2 is empty in this case. 

Hence, we may assume that $G$ is not 3-regular, that is, there exists a vertex $w$ in $G$ with degree at most 2. 
Let $G'$ be the graph obtained by deleting the vertex $w$ (together with all edges incident to it) from $G$. 
Then every vertex in $G'$ that was a neighbor of $w$ in $G$ has degree at most 2 in $G'$. 
According to the induction hypothesis, $G'$ has a representation as a Diophantine graph such that all vertices of degree at most 2 have different square-free part. 
In particular, if $w$ has degree 2 in $G$, then the two neighbors of $w$ are represented by numbers with different square-free part in the representation of $G'$. 
Now, we can apply Lemma~\ref{lem1}, \ref{lem2}, or \ref{lem3} if $w$ has degree 0, 1, or 2 in $G$, respectively. 
This yields an extension of the representation of $G'$ to a representation of $G$ such that the square-free part of the number associated to $w$ is different from the square-free part of all other numbers in the representation. 
Hence, in this representation of $G$, the square-free part of vertices of degree at most 2 are all different. 
\end{proof}

By using a similar argument, we can represent many small graphs as Diophantine graphs. 
The natural goal would be to show that all graphs with at most five vertices are Diophantine, except for $K_5$, of course. 
However, there is a graph on five vertices that we could neither represent as a Diophantine graph, nor could we prove that it is not Diophantine. 
It is the pyramid graph that consists of a 4-cycle together with an extra vertex linked to all four vertices of the 4-cycle. 
It can be obtained from $K_5$ by omitting two independent edges. 
Surprisingly, it is easy to see that the graph obtained from $K_5$ by omitting only one edge is Diophantine. 
We can simply consider a Diophantine quadruple, for instance, any member of the infinite parametrized family $(k-1,k+1, 4k, 16k^3-4k)$, and extend it by the upper regular extension $d_+$ of the three largest numbers; see Section~\ref{sec:remo}. 
This yields an infinite parametrized family of 5-tuples $(k-1,k+1, 4k, 16k^3-4k, 256k^5+256k^4-32k^3-64k^2+k+3)$ representing the graph obtained from $K_5$ by omitting an edge.  
Intuitively, one would expect that subgraphs of Diophantine graphs are also Diophantine: deleting an edge should pose a less demanding problem, as there are fewer expressions that need to be perfect squares. 
However, despite an extensive computer search, we have not found any representation of the pyramid graph. 

\begin{corollary}\label{cor:smalldio}
Let $G$ be a graph with at most five vertices that is not isomorphic to $K_5$ or the pyramid graph. 
Then $G$ is Diophantine. 
\end{corollary}
\begin{proof}
Let us say that a graph has a good Diophantine representation if all numbers in the representation has different square-free part. 
By using Lemmas~\ref{lem1}, \ref{lem2}, and \ref{lem3}, we obtain that all graphs with at most three vertices have a good Diophantine representation. 
Moreover, according to Lemmas~\ref{lem1}, \ref{lem2}, and \ref{lem3}, if there is a vertex $w$ of degree at most 2 in a graph $G$ and $G-w$ has a good Diophantine representation, then $G$ also has a good Diophantine representation. 
This yields a good Diophantine representation for all graphs with at most four vertices, except for $K_4$. 
However, $K_4$ also has such a representation, e.g., $\{1,3,8,120\}$. 

It remains to verify the assertion for five-vertex graphs $G$. 
If there is a vertex of degree at most 2 in $G$, then we can proceed as above. 
Thus we may assume that all degrees are at least three in $G$. 
Since the sum of degrees is even and $G\not\cong K_5$, the only possible degree sequences of $G$ are $(4,3,3,3,3)$ and $(4,4,4,3,3)$. 
These degree sequences uniquely determine the isomorphism type of $G$. 
The first one is the pyramid graph, which we excluded from the assertion (since we do not know whether it is Diophantine). 
The second one is $K_5$ with an edge deleted, which is represented by $\{1,3,8,120,11781\}$ (see the parametrized 5-tuple above with $k=2$). 
\end{proof}

\begin{proof}[Proof of Theorem \ref{thm1}]
Parts i), ii) and iii) of the statement are immediately given by Lemmas \ref{lem1}, \ref{lem2} and \ref{lem3}, respectively.
\end{proof}

\section{Proof of the results concerning the number of edges of Diophantine graphs}

In this section we prove our Theorems \ref{thm:densegr}, \ref{thm3} and \ref{thm4}, separating the cases of lower- and upper bounds.

\subsection{Proof of the lower bound}

As we mentioned already, Dujella \cite{Duj08} has shown that there are asymptotically $\frac{6}{\pi^2}N\log N$ edges in $D(V_N)$ as $N\rightarrow \infty$, where $V_N=\{1,\ldots, N\}$. 
In other words, the asymptotic edge density of $D(V_N)$ is $\frac{6}{\pi^2}\log N$. 
Dujella's proof relies on the following lemma (see \cite{vin}, exercise 9 on page 98), which we shall also need. 
Recall that $\omega(d)$ is the number of different prime divisors of an integer $d$. 

\begin{lemma}\label{lem:squaremod}
Let $S(a)$ be the number of solutions of the congruence $x^2\equiv 1 \pmod{a}$. Then $S(a)\leq 2^{\omega(a)+1}$. 
More precisely, 
\begin{enumerate}
\item if $2\nmid a$, then $S(a)=2^{\omega(a)}$, 
\item if $2\mid a$ but $4\nmid a$, then $S(a)=2^{\omega(a)-1}$,  
\item if $4\mid a$ but $8\nmid a$, then $S(a)=2^{\omega(a)}$, and
\item if $8\mid a$, then $S(a)=2^{\omega(a)+1}.$
\end{enumerate}
\end{lemma}

We can use a similar approach to that of Dujella's to estimate the degree of vertices in $D(V_N)$. 

\begin{lemma}
\label{lem:degree}
Let $1\leq a\leq N$. 
Then the degree of $a$ in $D(V_N)$ is at most $8\sqrt{N/a}\cdot 2^{\omega(a)}$.
\end{lemma}

\begin{proof}
An integer $1\leq b\leq N$ is linked to $a$ by an edge in $D(V_N)$ if and only if there is an integer $r$ such that $ab+1=r^2$. 
Thus the $\pmod a$ residue class of $r$ is the solution to the congruence $x^2\equiv 1 \pmod{a}$. 
Further, as $b\leq N$, we have $r\leq 2\sqrt{aN}$. 
Let us partition the numbers $\{1,\ldots, a\cdot \lceil 2\sqrt{aN}/a\rceil\}$ into blocks of consecutive integers of length $a$. 
The number of blocks is $\lceil 2\sqrt{aN}/a\rceil \leq 4\sqrt{N/a}$, and in each block, there are at most $2^{\omega(a)+1}$ solutions $r$ according to Lemma~\ref{lem:squaremod}. 
\end{proof}

Note that the coefficient $8$ in the estimate could be easily improved, but it is not going to play an important role. 
The degree of $a$ is always at least $\lfloor \sqrt{N/a}\rfloor\cdot 2^{\omega(a)-1}$, and for most numbers below $N$, it should be close to $\sqrt{N/a}\cdot 2^{S(a)}$.

\begin{corollary}\label{cor:totdeg}
Let $\delta>0$, $C>1$ and $t\in \mathbb{N}$. 
Then for $N$ large enough, the total degree in the graph $D(V_N)$ of all numbers $a$ in the interval $[N(\log N)^{-t\delta}, N(\log N)^{-(t-1)\delta}]$ such that $\omega(a)\leq C\log\log N$ is at most $8N(\log N)^{C\log 2 +\delta/2}$. 
\end{corollary}

\begin{proof}
There are at most $N(\log N)^{-(t-1)\delta}$ such numbers $a$. 
For each such $a$ we have $a\geq N(\log N)^{-t\delta}$, thus $\sqrt{N/a}\leq (\log N)^{t\delta/2}$. 
According to Lemma~\ref{lem:degree}, each of them has a degree at most $8(\log N)^{t\delta/2}\cdot 2^{C\log\log N}=8(\log N)^{C\log 2 + t\delta/2}$. 
Thus the total degree on the interval is bounded by the product $8N(\log N)^{(1-t/2)\delta+C\log 2} \leq 8N(\log N)^{C\log 2+\delta/2} $ for $N$ large enough.
\end{proof}

Since our argument to prove Theorem \ref{thm:densegr} is rather involved, we give some explanation of our strategy. 
We shall start out from $V_N$, and leave out vertices of ``small'' degree to increase the edge density $\Theta(\log N)$ of the graph $D(V_N)$. 
Note that the omission of a vertex $v$ from a graph $G$ increases the average degree if and only if $\deg(v)$ is less than the edge density, that is, half of the average degree in $G$. 
Hence, the natural idea would be to omit every vertex from $D(V_N)$ whose degree is less than $\frac{6}{\pi^2}\log N$. 
As we will see, asymptotically $100\%$ of the vertices have degree less than $\frac{6}{\pi^2}\log N$, and they cover asymptotically $0\%$ of the edges. 
More precisely, only $N/(\log N)^{c_1}$ vertices would remain after this operation with some $c_1>0$. 
This means that in the graph obtained, the average degree has order of magnitude $\Theta\left((\log N)^{1+c_1}\right)$ rather than $\log N$. 
Thus we could iterate the same process, and omit vertices with degree less than $(\log N)^{1+c_1}$ (approximately), only to realize that the average degree increases to the order of magnitude $\Theta\left((\log N)^{1+c_2}\right)$ with some $c_2>c_1$. 

One can heuristically determine the limit $c$ of the sequence $c_1, c_2, \ldots$, which is $c=2\log 2 -1$. 
(This will be clear from the proof of Theorem \ref{thm:densegr}.) 
Hence, the rough idea is to show that for any $\varepsilon>0$, asymptotically $100\%$ of the $\Theta(N\log N)$ edges survive if we omit vertices from $D(V_N)$ with degree at most $(\log N)^{2\log 2 - \varepsilon}$, and at the same time, only about $N(\log N)^{1-2\log 2 +\varepsilon}$ vertices remain. 
This yields a graph with edge density at least $(\log N)^{2\log 2 - \varepsilon}$ rather than $\Theta(\log N)$. 

To achieve our goal, we shall need a classical Hardy-Ramanujan type theorem from the literature. 
In \cite{HR17}, Hardy and Ramanujan showed that the normal order of $\omega(a)$ is $\log\log a$. 
Since then, the theorem was extended in many different directions; see \cite{HT88}. 
The variant we need was found by Sathe~\cite{Sat53} and (simplified by) Selberg~\cite{Sel54}. 
More precisely, we use the following consequence of their result, where $\pi(x,k)$ denotes the number of positive integers $a\leq x$ with $\omega(a)=k$.

\begin{theorem}[Sathe-Selberg]\label{thm:SatSel}
Let $C>1$ be a fixed constant. Then for all $x\geq 3$ we have 
$$
\sum\limits_{k>C\log\log x} \pi(x,k) \leq \frac{x}{\log x}\cdot \sum\limits_{k>C\log\log x} \frac{(\log\log x)^{k-1}}{(k-1)!}.
$$
\end{theorem}

The sum on the right can be estimated very accurately: it is the tail sum of the Taylor series of the exponential function with $\log\log x$ substituted into it. 
Note that two consecutive terms in the sum, namely the $k$-th and $(k+1)$-th terms have ratio $\frac{\log\log x}{k}<1/C$, thus the tail sum is dominated by a power series with quotient $1/C<1$. 
Hence, the tail sum is at most $\frac{C}{C-1}$ times the first term in the sum, that is, the term where $k=\lceil C\log\log x \rceil$ is substituted. 
This term can be computed asymptotically using Stirling's formula (i.e., $k!\sim \sqrt{2\pi k}(k/e)^k$), to obtain the following estimate. 
 
\begin{corollary}\label{cor:SaSe}
$$\sum\limits_{k>C\log\log x} \pi(x,k) < \frac{x}{\log x}\cdot \frac{C}{C-1}\cdot \frac{(e/C)^{C\log\log x}}{\sqrt{2\pi C\log\log x}} < x(\log x)^{C-C\log C-1}$$
for any fixed $C>1$ and $x$ large enough.
\end{corollary}

\begin{proof}[Proof of Theorem~\ref{thm:densegr}]
As the assertion only becomes stronger by decreasing $\varepsilon$, we may assume that $\varepsilon=1/S$ for some $S\in \mathbb{N}$, $S\geq 1000$. 

The differentiable function
$$
g(C):=C(1+\log 2)-C\log C -1
$$
has derivative $g'(C)=\log 2 - \log C$, hence $g$ is strictly monotonically increasing in the interval $[1,2]$. Moreover, we have $g(2)=1$. 
Thus there exists a $\delta>0$ such that $g(2-\varepsilon) + 5\delta < 1$. 
We may assume that $\delta=\varepsilon/(5R)=1/(5RS)$ for some $R\in \mathbb{N}$. 
Note that $R,S,\delta$ are fixed positive constants as soon as $\varepsilon>0$ is fixed.

We partition the vertices in $V_N$ into the following $10RS+1$ blocks; the first one is special, the rest follows a pattern indicated by the running index $t$: 
$$[1,N(\log N)^{-2}[,$$ 
$$[N(\log N)^{-(10RS)\delta},N(\log N)^{-(10RS-1)\delta}[,$$ 
$$[N(\log N)^{-(10RS-1)\delta},N(\log N)^{-(10RS-2)\delta}[, \ldots$$
$$[N(\log N)^{-t\delta},N(\log N)^{-(t-1)\delta}[, \ldots$$
$$[N(\log N)^{-\delta},N].$$

We keep all the numbers in the first, special block $[1,N(\log N)^{-2}[$. 
In the rest of the blocks, we only keep those numbers $a$ such that $\omega(a)>(2-\varepsilon)\log\log N$; the rest of the numbers are omitted from $V_N$. 
We define the constants $C_0=1.2=2 - 4RS\delta$, $C_1=2 - (4RS-1)\delta$, $\ldots$, $C_q=2 - (4RS-q)\delta, \ldots$, $C_{4RS-5R}=2 - \varepsilon$. 
We omit vertices from the last $10RS$ blocks iteratively: in step $0$, we omit those $a$ with $\omega(a)\leq C_0\log\log N$, and in every other step $q+1$, we omit those $a$ with $C_{q}\log\log N< \omega(a)\leq C_{q+1}\log\log N$. 
Note that by the end of this process, exactly the desired vertices are omitted. 

The total number of edges lost in step 0 is estimated by Corollary~\ref{cor:totdeg}. 
In each of the $10RS$ blocks, it is at most $8N(\log N)^{C_0\log 2 + \delta/2}$, totaling at most $80RSN(\log N)^{1.2\cdot\log 2 + \delta/2}< N(\log N)^{0.9}$ for $N$ large enough. 
Note that this is a negligible portion of the $\Theta(N\log N)$ edges.

The rest of the steps is treated similarly as the proof of Corollary~\ref{cor:totdeg}. 
The difference is that rather than using the trivial estimate for the number of vertices in an interval (i.e., the right endpoint of the interval), we apply Corollary~\ref{cor:SaSe}. 
Namely, in step $q+1$, we omit those $a$ from each of the $10RS$ intervals such that $C_{q}\log\log N< \omega(a)\leq C_{q+1}\log\log N$. 
The number of these vertices in the interval $[N(\log N)^{-t\delta}, N(\log N)^{-(t-1)\delta}[$ is clearly at most the number of those $a\leq y$ such that $C_{q}\log\log y< \omega(a)$, where $y=N(\log N)^{-(t-1)\delta}$. 
Hence, according to Corollary~\ref{cor:SaSe}, the number of such vertices is at most $y(\log y)^{C_q-C_q\log C_q-1}$. 
This yields the upper estimate 
$$N(\log N)^{-(t-1)\delta}\left(\log\left(N(\log N)^{-(t-1)\delta}\right)\right)^{C_q-C_q\log C_q-1}\leq $$
$$N(\log N)^{-(t-1)\delta}\left(\log N\right)^{C_q-C_q\log C_q-1} = N(\log N)^{C_q-C_q\log C_q-1-(t-1)\delta}$$
for the number of vertices omitted from the above interval in step $q+1$, for $N$ large enough. 
According to Lemma~\ref{lem:degree}, such a vertex $a$ has degree at most $8\sqrt{N/a}\cdot 2^{\omega(a)}$. 
Here, $a\geq N(\log N)^{-t\delta}$ making $N/a\leq (\log N)^{t\delta}$, and by assumption we have $\omega(a)\leq C_{q+1}\log\log N = C_{q+1}\log 2 \cdot \log_2\log N$. 
Thus all degrees are bounded by 
$$8(\log N)^{t\delta/2}(\log N)^{C_{q+1}\log 2} = 8(\log N)^{t\delta/2+(C_q+\delta)\log 2}.$$

Multiplying the two estimates by each other and by the number of intervals $10RS$ we obtain the bound 
$$160RSN(\log N)^{C_q-C_q\log C_q-1-(t-1)\delta+t\delta/2+(C_q+\delta)\log 2}=$$
$$160RSN(\log N)^{C_q(1+\log 2)-C_q\log C_q-1+(1+\log2-t/2)\delta}\leq$$
$$N(\log N)^{C_q(1+\log 2)-C_q\log C_q-1+2\delta} = N(\log N)^{g(C_q)+2\delta}$$
for the total number of lost edges in step $q+1$, for $N$ large enough. 
Taking the sum over all steps $1,2,\ldots,4RS-5R$, and by observing that all the $C_q$ are at most $2-\varepsilon$, and that the function $g(C)$ is strictly monotonically increasing on $[1,2]$, we have the upper estimate to the total number of omitted edges 
$$(4RS-5R)N(\log N)^{C(1+\log 2)-C\log C-1+2\delta}\leq $$
$$N(\log N)^{C(1+\log 2)-C\log C-1+3\delta}$$
with $C=2-\varepsilon$, for $N$ large enough. 
This expression clearly exceeds $N(\log N)^{0.9}$ for $N$ large enough (the exponent of $\log N$ is close to $1$), which is the upper estimate obtained in step 0. 
Hence, at the cost of an extra $\delta$ in the exponent, we can arrive at un upper estimate for the number of all omitted edges, including those omitted at step 0, namely 
$$N(\log N)^{C(1+\log 2)-C\log C-1+4\delta}$$
for $N$ large enough. 
As $C=2-\varepsilon$, due to the choice of $\delta$, this is at most $N(\log N)^{1-\delta}$. 
Thus only a negligible portion of the $\Theta(N\log N)$ edges were omitted throughout the process. 

The number of non-omitted vertices can be estimated by applying Corollary~\ref{cor:SaSe}. 
We kept all positive integers in the first, special interval $[1,N(\log N)^{-2}[$: this contributes at most $N(\log N)^{-2}$ vertices. 
From the rest of $V_N$, we only kept those vertices $a$ such that $\omega(a)>(2-\varepsilon)\log\log N$. 
For simplicity, we estimate their number from above by the number of those $a\leq N$ such that $\omega(a)>(2-\varepsilon)\log\log N$. 
According to Corollary~\ref{cor:SaSe}, $N(\log N)^{C-C\log C-1}$ with $C=2-\varepsilon$ is a correct upper estimate for this number, provided that $N$ is large enough. 
The function $h(C):=C-C\log C-1$ defined on $[1,2]$ has derivative $h'(C)=-\log(C)$, thus it is strictly monotonically decreasing on $[1,2]$. 
The derivative $h'(C)=-\log(C)$ is greater than $-0.7$ on the whole domain. 
Thus $h(2-\varepsilon)<h(2)+0.7\cdot\varepsilon=1-2\log 2+0.7\cdot\varepsilon$. 
Hence, in the non-special intervals, at most $N(\log N)^{1-2\log 2+0.7\cdot\varepsilon}$ vertices remain after the process. 
Together with the negligible amount $N(\log N)^{-2}$ in the first interval (note that $1-2\log 2+0.7\cdot\varepsilon> -0.4>-2$) it is still less than $N(\log N)^{1-2\log 2+0.8\cdot\varepsilon}$ vertices remaining, for $N$ large enough. 
As asymptotically $\frac{6}{\pi^2}N\log N$ edges remained, that is, more than $0.5\cdot N\log N$ for $N$ large enough, this yields an edge density in the remaining graph at least 
$$(0.5\cdot N\log N)/\left(N(\log N)^{1-2\log 2+0.8\cdot\varepsilon}\right) > (\log N)^{2\log 2-\varepsilon}$$
for $N$ large enough. 
As the number of remaining vertices $n$ is clearly at most $N$, we can conclude that the edge density in this $n$-element graph is at least $(\log n)^{2\log 2-\varepsilon}$.
\end{proof}

The constant $2\log 2\approx 1.386$ is optimal in the sense that the same proof idea cannot work with any larger constant. 
If we omit all (large enough) vertices $a$ from $V_N$ such that $\omega(a)>2\log\log N$, then the upper estimate for the total lost degree is $\Theta\left(N(\log N)^{g(2)}\right) = \Theta(N\log N)$, which is comparable to the number of edges. 
Of course, this does not mean that a modified proof could no longer work. 
Using a more refined result than Corollary~\ref{cor:SaSe} of the Sathe-Selberg theorem might clarify how many edges survive as $\omega(a)$ breaks the $2\log\log N$ barrier. 
If in that region edges disappear at a much slower rate than vertices, then the edge density could still increase by omitting larger degree vertices from the graph $V_N$. 

\subsection{Proof of the upper bounds}

We start with the proof of our unconditional result. 
In fact, the statement is a simple consequence of two seminal results.

\begin{proof}[Proof of Theorem \ref{thm3}]
By the main result of He, Togb\'e and Ziegler \cite{HTZ19} we know that $D(V)$ cannot contain a $K_5$. 
Thus the statement is a simple consequence of Tur\'an's theorem \cite{Tu41} with $k=5$.
\end{proof}

Now we give the proof of our conditional statement. 
For this, a deep result of Hindry and Silverman \cite{hs} will be vital. 
To its formulation we need to introduce some notation. 
Note that we shall simplify the general algebraic situation considered in \cite{hs} to the rational case, which is sufficient for our present purposes. 
Let $A,B$ be integers such that the discriminant $D_E=-4A^3-27B^2$ of the polynomial $x^3+Ax+B$ is not zero, and consider the elliptic curve
$$
E:\ \ \ y^2=x^3+Ax+B
$$
over $\mathbb Q$. 
Assume that the above model is minimal for $E$; that is, $|D_E|$ is minimal over the different models of $E$ as above. 
In other words, suppose that if $\ell$ is a positive integer with $\ell^4\mid A$ and $\ell^6\mid B$, then $\ell=1$. 
Further, let $S$ be a finite set of primes, and let ${\mathbb Z}_S$ be the set of rationals such that (in their primitive forms) all prime factors of their denominators belong to $S$. 
Finally, write $r_E$ and $C_E$ for the rank and conductor of $E$, respectively.

\begin{theorem}[Hindry and Silverman]
\label{thmhs}
There exists an absolute constant $C_0$ such that the number of points on $E$ from ${\mathbb Z}_S^2$ is at most
$$
C_0^{|S|+1+(1+r_E)\frac{\log |D_E|}{\log C_E}}.
$$
\end{theorem}

\begin{proof}
The statement is a simple consequence of Theorem 0.7 of \cite{hs}.
\end{proof}

\begin{proof}[Proof of Theorem \ref{thm4}]
We may assume that $t\geq 3$, otherwise we are done. 
Let $c$ be an element from the vertex class of $a$ and $b$, and $d$ be an element from the other vertex class.
Then with some integer $r$ we have
$$
(ad+1)(bd+1)(cd+1)=r^2.
$$
A straightforward calculation shows that
$$
(x,y)=(9abcd+3(ab+ac+bc),27abcr)
$$
is an integral point on the elliptic curve
\begin{equation}
\label{e}
E:\ \ \ y^2=x^3+Ax+B
\end{equation}
with
$$
A=-27a^2b^2+27a^2bc-27a^2c^2+27ab^2c+27abc^2-27b^2c^2,
$$
$$
B=27(2ab-ac-bc)(ab-2ac+bc)(ab+ac-2bc).
$$
We also obtain that the discriminant of $E$ is given by
$$
D_E=3^{12}a^2 b^2 c^2 (a-b)^2 (a-c)^2 (b-c)^2.
$$
Let $\ell$ be the largest positive integer such that $\ell^4\mid A$ and $\ell^6\mid B$. Then a simple calculation gives that if $p$ is a prime factor of $\ell$, then $p\mid 3ab(a-b)$. Writing $A=\ell^4 A_0$, $B=\ell^6 B_0$, $y=\ell^3 y_0$, $x=\ell^2 x_0$, we get
\begin{equation}
\label{e0}
E_0:\ \ \ y_0^2=x_0^3+A_0x_0+B_0
\end{equation}
as a minimal model of $E$. Further, we see that every integer point of $E$ belongs to a ${\mathbb Z}_S$-point of $E_0$, where $S$ consists of the prime divisors of $\ell$ - and hence $S$ is a subset of the set of prime divisors of $ab(a-b)$. Thus our claim follows from Theorem \ref{thmhs}, assuming the conjectures of Szpiro \eqref{szpiro} and N\'eron \eqref{neron}.
\end{proof}

\section{Proof of the results related to chromatic numbers of Diophantine graphs}

\begin{proof}[Proof of Theorem \ref{thm5}]
The verification of the statement is done by an exhaustive case distinction. 
The vertices are listed in a carefully chosen order, starting by the Diophantine quadruple $\{1,3,8,120\}$. 
As a starting step, we may assume that a four-coloring assigns the colors $0,1,2,3$ to these four vertices, respectively, thereby breaking the symmetry of the colors and reducing the number of possibilities by a factor of $24$. 
At all other vertices, initially the set of all colors $\{0,1,2,3\}$ is registered as possibilities. 
Whenever we assign a definite color to a vertex $u$, the algorithm (implemented in {\tt{Python}}) does a \emph{sweeping} operation: it deletes that color as a possible one from the color sets at all the neighbors of $u$. 
In particular, if such a set shrinks to a unique color, then its neighbors are also swept. 
The program iterates through the list of vertices, always making a case distinction when more than one colors can be assigned to the next vertex based on prior knowledge (following all available sweepings). 
Whenever this happens, the lists are copied as many times as there are possible assignable colors, and the program sweeps these copies of the graph again. 
If a vertex is found that has no possible colors left, the program deletes the corresponding copy of the graph. 
We have chosen $1000$ vertices initially, where there was a good chance for many edges in the induced Diophantine graph. 
Namely, we picked the $1000$ positive integers $a$ where the largest values of the function $2^{S(a)}/\sqrt{a}$ are attained; cf. Lemmas~\ref{lem:squaremod} and \ref{lem:degree} for justification. 
E.g., the maximum of this function is attained by the number $24$. 
The program verified that the $1000$-vertex Diophantine graph is not four-colorable, executing the computation in less than a second on an average PC. 
(All cases ended up being deleted after contradictions in the form of empty color sets at certain vertices were found.) 
There were never more than $200$ cases simultaneously. 

Then to reduce the size, we omitted low-degree vertices from this $1000$-vertex large graph: as discussed in the proof of Theorem \ref{thm:densegr}, as long as the degree of a vertex is less than the edge density, its omission increases the edge density. 
Sometimes this process halted, and we needed to omit somewhat larger degree vertices. 
Once we reduced down to a graph with $119$ vertices, where the program still verified that it cannot be colored by four colors, this greedy approach did not work anymore. 
So we started dropping vertices one-by-one, always checking that the resulting graph is still not four-colorable. 
Finally, we were left with the vertex set $V$ given in the statement.
\end{proof}

\section{Remarks and open problems}\label{sec:remo}

In this section we provide some remarks and open problems related to the topic discussed in the paper.

\vskip.2cm

\noindent{\bf A remark about Hamiltonian paths and cycles.} 
As $a(a+2)+1=(a+1)^2$, we have that $a,a+2\in V_N$ are always linked in $D(V_N)$. 
This simple observation shows that $D(V_N)$ is connected for all $N\geq 8$. 
Indeed, any pair of odd numbers and any pair of even numbers are connected by paths of the form $a, a+2, a+4, \ldots$, and the numbers $1$ and $8$ are also linked. 
Moreover, this observation almost yields a Hamiltonian path in $D(V_N)$ for $N\geq 8$. 
Starting from the largest odd number, we can walk down the odd numbers by jumps of distance two, all the way to the number 1. 
We follow up with the walk $8, 6, 4, 2, 12$, and then we increase the number by two at every turn. 
This path only avoids the number 10.  

Somewhat surprisingly, there is never a Hamiltonian cycle in $D(V_N)$. 
We show this assertion by partitioning the numbers in $V_N$ into $\pmod 4$ residue classes. 
Note that elements of the class of 2 are only linked to numbers divisible by four, and the former class is always at least as big as the latter (and nonempty if $N\geq 2$). 
Thus in a Hamiltonian cycle, both neighbors of a number in the class of $2$ should be divisible by four, and this can only occur if representatives of the two classes alternate in the cycle. 
However, this leaves no room for odd numbers to be included in the cycle. 

So there is always a path that is one short of being a Hamiltonian path, and there is never a Hamiltonian cycle in $D(V_N)$. 
This highlights the problem of the existence of a Hamiltonian path in $D(V_N)$. 
A similar argument as above for cycles shows that there cannot be a Hamiltonian path in $D(V_N)$ if $N$ is congruent to 2 or 3  $\pmod 4$, since in that case there are more elements in the class of 2 than there are numbers in $V_N$ divisible by four. 
In some small examples where $N$ is congruent to 0 or 1  $\pmod 4$, it is easy to verify that there is no Hamiltonian paths: these small examples are $N=17, 32, 33$. 
However, for other small values of $N$ in the residue classes 0 or 1, we have always found Hamiltonian paths in $D(V_N)$, and it seems to be a condition that is more and more probable to hold as $N$ gets larger. 
It is also easy to show that there is a Hamiltonian path in $D(V_N)$ for infinitely many values of $N$: e.g., this is the case for $N=16k^2$ for all $k\in\mathbb{N}$. 
(As witnessed by the path $4k^2-1$, $4k^2-3$, $\ldots, 3, 1, 16k^2-1$, $16k^2-3$, $\ldots, 4k^2+3, 4k^2+1, 16k^2$, $16k^2-2$,  $\ldots, 4, 2$.)

\begin{problem}
Is there a threshold $N_0\in \mathbb{N}$ such that for all $N\geq N_0$, $N\equiv 0,1 \pmod{4}$ there is a Hamiltonian path in $D(V_N)$?
\end{problem}

\vskip.2cm

\noindent{\bf A remark about extensions of Diophantine graphs.} 
As a simple consequence of Theorem \ref{thm1}, we easily obtain that every graph with maximum degree at most two is Diophantine. 
On the other hand, as $K_5$ is not Diophantine by \cite{HTZ19}, there exist non-Diophantine graphs with maximum degree $k$ for any $k\geq 4$. 
So the following question arises naturally.

\begin{problem}
Is it true that every finite graph with maximum degree at most three is Diophantine? 
Or equivalently, is every finite 3-regular graph Diophantine? (Cf. Corollary~\ref{cor:3reg}.)
\end{problem}

In fact, we expect a negative answer. 
The smallest 3-regular graph is $K_4$, which is Diophantine. 
There are no 3-regular graphs on five vertices, so the second smallest examples have six vertices. 
Up to isomorphism, there are exactly two 3-regular graphs on six vertices. 
One of them is the complement of a 6-cycle, which is Diophantine, witnessed by the representation $\{1,3,8,10,96,168\}$. 
The other one is the complete bipartite graph $K_{3,3}$. 
That is the smallest open case of the problem; see also \cite{Dujproblems} as mentioned before. 
Besides the pyramid graph on five vertices, $K_{3,3}$ is in fact the smallest graph $G$ such that it is unknown whether $G$ is Diophantine. 

\begin{problem}
Is the pyramid graph Diophantine?
\end{problem}

Another question related to Theorem \ref{thm1} is the following.

\begin{problem}
Is it true that every Diophantine graph can be represented as $D(V)$, with some $V$ consisting of integers having pairwise different square-free part?
\end{problem}

\vskip.2cm

\noindent{\bf A remark about upper bounds for the number of edges of Diophantine graphs.} 
As it was pointed out earlier, bipartite graphs are crucial in this respect. 
So we propose the following question (for which we strongly expect a negative answer).

\begin{problem}
Is $K_{t,t}$ a subgraph of a Diophantine graph for all $t\in \mathbb{N}$?
\end{problem}

Once again, $K_{3,3}$ is the smallest open case of this problem. 
Considering the fact that even for a graph as simple as $K_{3,3}$, the question whether it is Diophantine seems hard to answer (and the question for $K_{t,t}$ is related to deep conjectures), we propose a problem concerning the complexity of the language of Diophantine graphs. 

\begin{problem}
Is it decidable whether an input finite graph is Diophantine? 
\end{problem}

It is well-known that the solvability of Diophantine equations is undecidable, see \cite{DPR61,Mat70}. 

\vskip.2cm

\noindent{\bf Remarks about the chromatic number of Diophantine graphs.} 
By a classical result of Bollob\'as \cite{Bol} we know that the chromatic number of a $p$-random graph on $n$ vertices is asymptotic to $n/\log_q n$ with $q=1/(1-p)$, so it goes to infinity with $n$. 
Based on this theorem, our first idea was that the chromatic number of Diophantine graphs should not be bounded: it seems plausible that  $D(V_N)$ should contain large quasirandom regular subgraphs with any fixed degree if $N$ is large enough. 
However, somewhat surprisingly, if we modify the definition of the graph relation by linking $a$ and $b$ if and only if $ab+2$ is a perfect square, then the graph obtained on the vertex set $\mathbb{N}$ is 3-colorable. 
Indeed, we can simply color even numbers red, those of the form $4k+1$ blue, and those of the form $4k+3$ green. 
As the $\pmod{4}$ residue of a square is 0 or 1, this is a valid coloring. 
Analyzing the phenomenon, one finds that such a construction using residue classes as colors is only possible if the number added to $ab$ in the definition is congruent to $2$ modulo $4$, as every other number is the difference of two squares. 
Still, we consider this observation as a warning. 
Going back to the standard definition, where $a$ and $b$ are linked by an edge if and only if $ab+1$ is a perfect square, it is hard to find an apparent reason why the chromatic number of any Diophantine graph could not be at most four. 
(The 80-vertex counterexample was found by an extensive computer-assisted calculation.) 
As there are Diophantine quadruples such as $\{1,3,8,120\}$, four colors are needed in general. 
The smallest graph that is not 4-colorable is $K_5$, which is not Diophantine. 
Moreover, the smallest $K_5$-free graph that is not 4-colorable has seven vertices: it consists of a five-cycle and two additional vertices that are linked to every other vertex. 
We do not know whether this graph $G$ is Diophantine. 
However, provided that the Regular Diophantine Quadruple Conjecture holds, it is not. 
To formulate the conjecture, we need to introduce some notation. 
Let $\{a,b,c\}$ be a Diophantine triple, write
$$
ab+1=r^2,\ \ \ ac+1=s^2,\ \ \ bc+1=t^2
$$
and set
$$
d_\pm:=a+b+c+2abc\pm 2rst.
$$
Then Diophantine quadruples of the form $\{a,b,c,d_\pm\}$ are called regular. 
The Regular Diophantine Quadruple Conjecture says that every Diophantine quadruple is regular. 
Note that it easily implies that any Diophantine triple $\{a,b,c\}$ can be extended to a Diophantine quadruple in at most two ways. 
(Assuming $0<a<b<c$, we have $0\leq d_-<c<d_+$.) 
For more details about the Regular Diophantine Quadruple Conjecture, see e.g. the homepage \cite{Dujhomepage} and the references given there. 
Now we show how this conjecture implies that the above $G$ with seven vertices cannot be a Diophantine graph. 
Let $u,v$ denote the two vertices that are linked to all the other vertices, and let $a,b,c,d,e$ denote the other five vertices along the five-cycle so that $a$ is the smallest out of the five and $(a,b,c,d,e)$ is a path. 
Then the Diophantine triple $(u,v,b)$ is extended to Diophantine quadruples in two different ways, namely by $a$ and by $c$. 
Assuming that the conjecture holds, these extensions are the two regular extensions. 
As $a<b$, $a$ cannot be the larger out of the two regular extensions of $(u,v,b)$, thus it is $c$. 
In particular, $c$ is the largest in the quadruple $(u,v,b,c)$, and consequently, $b<c$.  
We can repeat the same argument, obtaining that $d$ is the largest in the quadruple $(u,v,c,d)$, $e$ is the largest in the quadruple $(u,v,d,e)$, and finally, $a$ is the largest in the quadruple $(u,v,e,a)$. 
This is a contradiction, since $a$ is smallest in the five-cycle.

Using an approach which is similar to that used in the proof of Theorem \ref{thm5}, we tried to find a Diophantine graph that is not five-colorable. 
The complexity of this problem is orders of magnitude larger than that of the four-colorable variant. 
We started with the $1000$ numbers as earlier, omitted the low-degree vertices, and extended them by numbers below a million that have many neighbors among them. 
A much denser graph can be obtained in this way (with $1000$ vertices and average degree nearly $20$). 
It is almost certain that the graph is not five-colorable: we have done the same case distinction as in the proof of Theorem \ref{thm5}, as long as the memory of the computer allowed it, and then tried random extensions rather than doing a full case distinction. 
The program usually found a contradiction after coloring the first $60$ vertices, and it always halted before reaching the $100$-th vertex on the list. 
However, a complete verification seems to be beyond the computational horizon. 
Based on the above random samplings, we estimate that the complete runtime of the program would be around $10^{13}$ secs. 
Even though the task can be distributed almost perfectly among different computers, it would take about three years to execute it on $100000$ cores. 
Perhaps this runtime can be pushed down below the computational horizon by a more carefully chosen vertex set.

As the edge density of a Diophantine graph can be arbitrarily large (and in particular, the same applies to the minimum degree), there is no obvious upper bound for the chromatic number of these graphs, either. 
So we suggest the following question.

\begin{problem}
Is there a Diophantine graph that is not five-colorable? More generally, are there Diophantine graphs with arbitrarily large chromatic number?
\end{problem}

\section{Acknowledgments}
The authors are grateful to M\'arton Szikszai for the fruitful discussions. 
This research was supported in part by the HUN-REN Hungarian Research Network and the NKFIH grants 128088, 130909 and KKP~138270.

\end{document}